\documentclass[10pt]{article}
\usepackage[utf8]{inputenc}
\usepackage[T1]{fontenc}
\usepackage{lmodern}
\usepackage{graphicx}
\newcommand{\vs}{\vskip5pt}

\usepackage{amsmath}% http://ctan.org/pkg/amsmath
\usepackage{amsthm}
\usepackage{amssymb}
\usepackage[export]{adjustbox}% http://ctan.org/pkg/adjustbox
\usepackage[margin=0.5in]{geometry}
\usepackage{setspace}
\usepackage{enumitem}
\usepackage{float}
\usepackage{tabularx,ragged2e}
\newtheorem{thm}{Theorem}[section]
\newtheorem{defn}[thm]{Definition}
\newtheorem{prop}[thm]{Proposition}

\newtheorem{rem}[thm]{Remark}
\newtheorem{lem}[thm]{Lemma}
\newtheorem*{lem*}{Lemma}
\newtheorem*{thm*}{Theorem}
\newtheorem*{cor*}{Corollary}
\newtheorem*{rem*}{Remark}
\newtheorem{clm}[thm]{Claim}
\usepackage[all]{xy}
\usepackage{tikz-cd}
\usepackage{mathtools}
\usepackage{comment}

\newcommand{\N}{\mathbb{N}}
\newcommand{\G}{\Gamma}

\usepackage{hyperref}
\hypersetup{
	colorlinks=true,
	linkcolor=red,
	citecolor=blue,
	pdftitle={Hyperfiniteness of the boundary actions of relatively hyperbolic groups},
	pdfpagemode=FullScreen,
}

\title{Hyperfiniteness of boundary actions of relatively hyperbolic groups}
\date{}
\author{Chris Karpinski\footnote{McGill University. Email: christopher.karpinski@mail.mcgill.ca.}}

\begin{document}
	
	\maketitle

	\begin{abstract}
		We show that if $G$ is a finitely generated group hyperbolic relative to a finite collection of subgroups $\mathcal{P}$, then the natural action of $G$ on the geodesic boundary of the associated relative Cayley graph induces a hyperfinite equivalence relation. As a corollary of this, we obtain that the natural action of $G$ on its Bowditch boundary $\partial (G,\mathcal{P})$ also induces a hyperfinite equivalence relation. This strengthens a result of Ozawa obtained for $\mathcal{P}$ consisting of amenable subgroups and uses a recent work of Marquis and Sabok. 
	\end{abstract}
%\tableofcontents

%\newpage

\section{Introduction}

	 %In recent years, a problem of interest lying at the confluence of geometric group theory and descriptive set theory has been the question of whether certain classes of groups induce hyperfinite Borel equivalence relations via natural actions on their boundaries.
	
	This paper studies equivalence relations induced by boundary actions of relatively hyperbolic groups. The study of boundary actions began with the work of Connes, Feldman and Weiss in \cite{connes_feldman_weiss_1981} and Vershik in \cite{Vershik78} who studied the actions of free groups on their boundaries. They showed that for a free group, its action on the Gromov boundary is $\mu$-hyperfinite for every Borel quasi-invariant probability measure $\mu$ on the boundary. Adams \cite{ADAMS1994765} later generalized this result to all hyperbolic groups.
	
	\vs
	
	Relatively hyperbolic groups were introduced by Gromov \cite{Grom}; see also the monograph of Osin \cite{Osin06}. Given a relatively hyperbolic group $G$ with a collection of parabolic subgroups $\mathcal{P}$ there is a natural boundary called the Bowditch boundary, denoted $\partial(G, \mathcal{P})$, which is a compact metrizable space on which $G$ acts naturally by homeomorphisms. 
	
	\vs
	
	In \cite{Ozawa}, Ozawa generalized the work of Adams \cite{ADAMS1994765} to the actions of relatively hyperbolic groups on their Bowditch boundary under the assumptions that the parabolic subgroups are exact. When the parabolic subgroups of $G$ in $\mathcal{P}$ are amenable, Ozawa \cite{Ozawa} proved that the action of $G$ on $\partial(G, \mathcal{P})$ is topologically amenable, and, more generally, when the parabolic subgroups are exact, Ozawa \cite{Ozawa} proved that the group $G$ is exact. Alternative proofs of the exactness of the group were given by Osin \cite{Osin02} who worked with parabolic subgroups with finite asymptotic dimension and by Dadarlat and Guentner \cite{DG} who worked with parabolic subgroups that are uniformly embeddable into a Hilbert space. 
	
	\vs
	
	In \cite{Zimmer78}, Zimmer introduced the notion of amenability of equivalence relations; see also the work of Connes, Feldman and Weiss \cite{connes_feldman_weiss_1981}. By \cite[Theorem 5.1]{Adams94}, a measurable action of a countable group $G$ on a standard probability space $(X, \mu)$ is $\mu$-amenable if and only if $\mu$-almost all stabilizers are amenable and the orbit equivalence relation is $\mu$-amenable. 
	
	\vs
	
	In this paper we generalize the result of Ozawa and work with relatively hyperbolic groups without any assumptions on the parabolic subgroups. In fact, we consider boundary actions from the Borel perspective. A countable Borel equivalence relation is called \emph{hyperfinite} if it is a countable increasing union of finite Borel sub-equivalence relations. Dougherty, Jackson and Kechris showed in \cite[Corollary 8.2]{DJK94} that the boundary action of any free group induces a hyperfinite orbit equivalence relation. The result of Dougherty, Jackson and Kechris was generalized to cubulated hyperbolic groups by Huang, Sabok and Shinko in \cite{HSS19}, and later to all hyperbolic groups by Marquis and Sabok in \cite{MS20}. In this paper, we prove the following:
	
	\newtheorem{theorem}{Theorem}
	\renewcommand*{\thetheorem}{\Alph{theorem}}
	\newtheorem{corollary}[theorem]{Corollary}
	\renewcommand*{\thecorollary}{\Alph{theorem}}
	
	\begin{theorem}
		\label{A}
		Let $G$ be a finitely generated group hyperbolic relative to a finite collection of subgroups $\mathcal{P}$ and let $\hat{\G}$ be the associated relative Cayley graph. Then the natural action of $G$ on the geodesic boundary $\partial \hat{\G}$ induces a hyperfinite orbit equivalence relation. 
	\end{theorem}
	
		\begin{corollary}
		\label{B}
		Let $G$ be a finitely generated group hyperbolic relative to a finite collection of subgroups $\mathcal{P}$. Then the natural action of $G$ on the Bowditch boundary $\partial (G,\mathcal{P})$ induces a hyperfinite orbit equivalence relation. 
	\end{corollary}
	
	Corollary \ref{B} in particular strengthens the result of Ozawa \cite{Ozawa} in case the parabolic subgroups are amenable. Indeed, hyperfiniteness implies $\mu$-amenability for every invariant Borel probability measure $\mu$ and by \cite[Theorem 3.3.7]{ADR}, an action of a countable group on a locally compact space by homeomorphisms is topologically amenable if and only if it is $\mu$-amenable for every invariant Borel probability measure $\mu$.
	
	\vs
 	
	We proceed by following a similar approach to \cite{HSS19} and \cite{MS20}, studying \textit{geodesic ray bundles} $\text{Geo}(x, \eta)$ in relative Cayley graphs (Definition \ref{2.1}). For the case of a cubulating hyperbolic group $G$ studied in \cite{HSS19}, the crucial property from which the hyperfinitess of the boundary action of $G$ follows is the finite symmetric difference of geodesic ray bundles: for any $x,y \in G$ and any $\eta \in \partial G$, $\text{Geo}(x, \eta) \triangle \text{Geo}(y, \eta)$ is finite (see \cite[Theorem 1.4]{HSS19}). In \cite{Touikan2018OnGR}, Touikan showed that this symmetric difference need not be finite in Cayley graphs of general hyperbolic groups, although in \cite{Marquis_2018}, Marquis provides many examples of groups acting geometrically on locally finite hyperbolic graphs where this finite symmetric difference property does hold. In \cite{MS20}, Marquis and Sabok define a modified version of the geodesic ray bundle, denoted $\text{Geo}_1(x, \eta)$ for $x \in G$ and $\eta \in \partial G$ (see \cite[Definition 5.5]{MS20} and Definition \ref{2.7} in our paper) and show (\cite[Theorem 5.9]{MS20}) that these modified geodesic ray bundles satisfy a finite symmetric difference property: $\vert \text{Geo}_1(x, \eta) \triangle \text{Geo}_1(y, \eta) \vert < \infty$ for each $x, y \in G$ and for each $\eta \in \partial G$. Marquis and Sabok then deduce hyperfiniteness of the boundary action as a consequence of this finite symmetric difference property of the modified bundles (see \cite[Section 6]{MS20}). %Although the finite symmetric difference condition for geodesic ray bundles is restrictive, in \cite{Marquis_2018}, Marquis provides many examples of groups acting geometrically on hyperbolic graphs where this finite symmetric difference property does hold, thereby establishing hyperfiniteness of the boundary action. 
	
	\vs 
	
	Local finiteness of the Cayley graph plays a crucial role in establishing the finite symmetric difference property of the $\text{Geo}_1$ bundles in \cite{MS20}. However, relative Cayley graphs of relatively hyperbolic groups are not locally finite. To make up for this loss of local finiteness, we rely on finiteness results about relative Cayley graphs of relatively hyperbolic groups from \cite{Osin06} (namely, \cite[Theorem 3.26]{Osin06}).

	\vs
	
	We note also that the hyperfiniteness of boundary actions has been studied beyond relatively hyperbolic groups. Przytycki and Sabok have recently established the hyperfiniteness of the actions of a mapping class group of an oriented surface of finite type on the boundaries of the arc graph (\cite[Theorem 1.1]{przytycki_sabok_2021}) and the curve graph (\cite[Corollary 1.2]{przytycki_sabok_2021}) of the surface.
	
	%This paper is outlined as follows. In Section \ref{2}, we cover preliminaries. Then in Section \ref{3}, we prove a crucial finite sections property of geodesic ray bundles in relative Cayley graphs (Theorem \ref{3.1}, which is the main source of new content in this paper) that yields the local finiteness of these bundles, and we establish the finite symmetric difference property of the modified ($\text{Geo}_1$) bundles (Theorem \ref{3.7}). In Section \ref{4}, we show the hyperfiniteness of the action of $G$ on $\partial \hat{\G}$ as a consequence of the finite symmetric difference property of $\text{Geo}_1$ bundles, closely following the approach in Section 6 of \cite{MS20}. The main difference between Section \ref{4} of our paper and Section 6 of \cite{MS20} is our coding of labels of geodesic rays in the non locally finite relative Cayley graph $\hat{\G}$. We conclude by showing Corollary \ref{B}, which follows immediately from Theorem \ref{A}.

	%The main difference between this paper and \cite{MS20} is in Section \ref{3} where we prove the crucial finite sections property of geodesic ray bundles in relative Cayley graphs (Theorem \ref{3.1}, whose proof is the main source of new content in this paper). Equipped with the results of Section \ref{3}, in Section \ref{4} we show the hyperfiniteness of the action of $G$ on $\partial \hat{\G}$ which follows almost exactly as in Section 6 of \cite{MS20} where the hyperfiniteness of the boundary action of a hyperbolic group is shown as a consequence of the finite symmetric difference property of $\text{Geo}_1$ bundles. 
	
	\vs
	
	\textbf{Acknowledgement}: I owe great thanks to my advisor Marcin Sabok for his continuous support, patience and guidance throughout the production of this work. 
	
	\section{Preliminaries}
	\label{2}
	
	In this paper, for a hyperbolic metric space $X$, $\partial X$ will denote the geodesic boundary of $X$. We will also denote $C_{hb}(X)$ the \emph{horoboundary} of $X$ (see \cite[Section 2.4]{MS20} for a definition of the horoboundary). 
	
	\subsection{Relatively hyperbolic groups}
	\label{2.2}
	
	Relatively hyperbolic groups were first introduced by Gromov in his seminal paper \cite{Grom} as a generalization of hyperbolic groups. The following definitions can be found in \cite{Osin06}.
	
	\vs
	
	Let $G$ be a group generated by a finite set $X$,  let $\mathcal{P} = \{H_1,...,H_n\}$ be a collection of subgroups of $G$ and let $\mathcal{H} = \bigcup \mathcal{P}$. The \textbf{relative Cayley graph} associated to $X$ and $\mathcal{P}$ is the Cayley graph $\hat{\G}$ with respect to the generating set $X \cup \mathcal{H}$. This graph can  be identified with the \textbf{coned-off Cayley graph} obtained by starting with the Cayley graph $\G$ of $G$ with respect to $X$, adjoining to $\G$ a vertex $v_{gH_i}$ for each left coset $gH_i$ and connecting each vertex of $gH_i$ in $\G$ to $v_{gH_i}$ by an edge of length $\frac12$. The notation $d_X$ and $d$ refer to the word metrics with respect to the generating sets $X$ and $X \cup \mathcal{H}$, respectively. We will use the notation $B_r^X(x)$ to denote the closed ball of radius $r$  in the metric $d_X$ about the point $x \in G$.
	
	\vs
	
	A finitely generated group $G$ is \textbf{hyperbolic relative to} a collection of subgroups $\mathcal{P} = \{H_1,...,H_n\}$ if there exists a finite generating set $X$ of $G$ such that the associated relative Cayley graph is hyperbolic and satisfies the bounded coset penetration property (BCP) (see \cite[Definition 6.5]{Osin06} for the definition of the BCP; we will not need to use the definition of BCP, so we do not define it here). Relative hyperbolicity is invariant under change of finite generating set by \cite[Proposition 2.8]{Osin06}.

	\vs 
	
	%Let $\alpha, \beta$ be a pair of $(\lambda, \varepsilon)$-geodesics in $\hat{\G}$ for $\lambda \geq 1$ and $\varepsilon \geq 0$. We call $\alpha, \beta$ $\mathbf{k}$\textbf{-similar} (for $k \geq 0$)  if $\max\{d_X(\alpha_-, \beta_-), d_X(\alpha_+, \beta_+)\} \leq k$. We call $\alpha, \beta$ \textbf{symmetric} if they have the same label. Given vertices $u$ on $\alpha$ and $v$ on $\beta$, we say that $u,v$ are \textbf{synchronous vertices} if $\ell([\alpha_-, u]) = \ell([\beta_-, v])$ (where $[\alpha_-, u]$ and $[\beta_-, v]$ denote the segments along $\alpha, \beta$, respectively between $\alpha_-, u$ and $\beta_-, v$, respectively).  Given two  $H_i$-components $p,q$ of $\alpha, \beta$, respectively, we say that $p,q$ are \textbf{synchronous components} if the vertices $p_-, q_-$ are synchronous. 
	
	%Given a set $S$, we will denote $S^*$ the free monoid generated by $S$. For a path $p$ in $\hat{\G}$ or $\gamma$, we will denote the label of $p$ by $\phi(p) \in (X \cup \mathcal{H})^*$, and given a word $w \in (X \cup \mathcal{H})^*$, we will denote the element in the group $G$ represented by $w$ as $\pi(w)$. 
	For a finitely generated group $G$ hyperbolic relative to a finite collection $\mathcal{P}$ of subgroups, there is a natural compact metrizable space on which $G$ acts naturally by homeomorphisms, denoted $\partial (G,\mathcal{P})$ and called the \textbf{Bowditch boundary} (see \cite[Section 4]{Bow12} for the construction of the Bowditch boundary). The following theorem is the main ingredient in establishing Corollary \ref{B} as a result of Theorem \ref{A}. 
		
		\begin{thm}
			\label{2.bb}
			Let $G$ be hyperbolic relative to a finite collection of subgroups $\mathcal{P}$, with relative Cayley graph $\hat{\G}$. Then $\partial \hat{\G}$ embeds $G$-equivariantly and homeomorphically into $\partial (G, \mathcal{P})$ with countable complement. 
		\end{thm}
	
		\begin{proof}
			
			In \cite[Proposition 1, Section A.2]{Dah03}, it is shown that the coned-off Cayley graph $\hat{\G}$ witnesses the relative hyperbolicity of $G$ with respect to $\mathcal{P}$ according to Definition 2 of relative hyperbolicity from \cite{Bow12}. Therefore, by \cite[Proposition 8.5]{Bow12} and \cite[Proposition 9.1]{Bow12}, $\partial \hat{\G}$ embeds $G$-equivariantly and homeomorphically into $\partial (G, \mathcal{P})$ and $\partial \hat{\G}$ has countable complement in $\partial (G, \mathcal{P})$. 
			
		\end{proof}
	
	\subsection{Combinatorial Geodesic Ray Bundles}
	
	Let $X$ be a hyperbolic graph equipped with its natural combinatorial metric (assigning edges length 1), and denote the vertex set of $X$ by $X^{(0)}$. We present some definitions and terminology used in \cite{MS20} that we will use in our paper. We refer the reader to Sections 3 and 4 of \cite{MS20} for a further study of the objects we define in this section. 
	
	\begin{defn}
		\label{2.1}
	For $x \in X^{(0)}$ and $\eta \in \partial X$, define $CGR(x, \eta)$ to be the set of all combinatorial geodesic rays (CGRs) based at $x$ and  define the \textbf{combinatorial geodesic ray bundle} $\text{Geo}(x, \eta) = \bigcup \text{CGR}(x, \eta)$ to be the set of all vertices on CGRs in $\text{CGR}(x, \eta)$.
	\end{defn}

	By \cite[Lemma 3.2]{MS20}, every CGR $\gamma = (x_n)_n$ converges to some $\xi \in C_{hb}(\hat{\G})$. We denote such limit $\xi = \xi_{\gamma}$. 
	
	\begin{defn}
		
		Fixing a basepoint $z \in X^{(0)}$, for $\eta \in \partial \hat{\G}$ define the \textbf{limit set} $\Xi(\eta) = \{\xi_{\gamma} : \gamma \in \text{CGR}(z, \eta)\}$. 
	\end{defn}
	
	By \cite[Lemma 3.1]{MS20} (which says that we can move the basepoint of any geodesic ray to any other basepoint to obtain a geodesic with the same tail), the definition of $\Xi(\eta)$ is independent of the basepoint (i.e. for any $z_1, z_2 \in X^{(0)}$ and $\xi \in \Xi(\eta)$, we have $\xi = \xi_{\gamma}$ for some $\gamma \in \text{CGR}(z_1, \eta)$ if and only if $\xi = \xi_{\gamma'}$ for some $\gamma' \in \text{CGR}(z_2, \eta)$). 

	\begin{defn}
			\label{2.4}
		For $x \in X^{(0)}, \eta \in \partial X$ and $\xi \in \Xi(\eta)$, define the \textbf{combinatorial sector} $Q(x, \xi) = \{y \in X^{(0)} : y \in \gamma \text{ for some } \gamma \in \text{CGR}(x, \eta) \text{ with } \xi_{\gamma} = \xi\}$. 
	\end{defn}
	
	\begin{defn}
		\label{2.5}
		For $\eta \in \partial X$, a vertex $x \in X^{(0)}$ is $\mathbf{\eta}$\textbf{-special} if $\bigcap_{\xi \in \Xi(\eta)}Q(x, \xi)$ contains a CGR $\gamma$. The set of all $\mathbf{\eta}$\textbf{-special} vertices is denoted $X_{s,\eta}$.
	\end{defn}

	By \cite[Lemma 4.7]{MS20}, if $x \in X_{s,\eta}$, then there exists a unique $\xi' \in \Xi(\eta)$ such that $\bigcap_{\xi \in \Xi(\eta)}Q(x, \xi) = Q(x, \xi')$. We denote such $\xi'$ by $\xi' = \xi_{x, \eta}$. 

	\vs

	Our main objects of interest will be the following modified geodesic ray bundles, first defined in \cite[Definition 5.5]{MS20}. 
	
	\begin{defn}
			\label{2.7}
		Let $x \in X^{(0)}$ and $\eta \in \partial X$. For $\xi \in \Xi(\eta)$, let $Y(x, \xi)$ be the set of $\eta$-special vertices $y \in \text{Geo}(x, \eta)$ with $\xi_{y, \eta} = \xi$ at minimal distance to $x$. Put
		
		\begin{equation*}
			\text{Geo}_1(x, \eta) = \bigcup_{\xi \in \Xi(\eta)} \bigcup_{y \in Y(x, \xi)} Q(y, \xi)
		\end{equation*}
	\end{defn}

	%Note that $Y(x, \xi)$ may be empty when $X$ is not locally finite or when $X$ is bounded, so it is possible in general for $\text{Geo}_1(x, \eta)$ to be empty. 
	
	\section{Geodesic Ray Bundles in Relatively Hyperbolic Groups}
	\label{3}
	In this section, we examine modified geodesic ray bundles in the relative Cayley graph $\hat{\G}$ and prove that these modified bundles have finite symmetric difference for a fixed boundary point. This section generalizes \cite[Theorem 5.9]{MS20}. 
	
	\vs 

	We begin by showing that $\vert \{\gamma(i) : \gamma \in \text{CGR}(x, \eta)\} \vert$ is uniformly bounded for each $i$, each $x \in G$ and each $\eta \in \partial \hat{\G}$, which is a well-known property in any uniformly locally finite hyperbolic graph.
	
	\vs 
	
	We will make use of the following result, which states that geodesic triangles in the relative Cayley graph are slim with respect to the metric $d_X$ for some finite generating set $X$.
	
	\begin{thm}
		\label{3.0}
		Let $G$ be a finitely generated group hyperbolic relative to a collection of subgroups $\{H_1,...,H_n\}$. There exists a finite generating set $X$ of $G$ such that the following holds. There exists a constant $\nu$ such that for any geodesic triangle $pqr$ in the relative Cayley graph $\hat{\G}$ and any vertex $u$ on $p$, there exists a vertex $v$ on $q \cup r$ such that $d_X(u,v) \leq \nu$. 
	\end{thm}

	\begin{proof}
		The finite generating set $X$ is constructed in the proof of  \cite[Lemma 3.1]{Osin06} and it is shown in the proof of \cite[Theorem 3.26]{Osin06} that $X$ satisfies the stated property.
	\end{proof}
	
	%Note that Theorem \ref{3.0} implies that geodesic triangles in the relative Cayley graph $\hat{\G}$ are $\nu$-slim (because $d \leq d_X$), so that $\hat{\G}$ is $\nu$-hyperbolic. 

	Here is the main result of this section. 
	
	\begin{thm}
		\label{3.1}
		Let $G$ be a finitely generated group hyperbolic relative to a collection of subgroups $\{H_1,...,H_n\}$. There exists a finite generating set $X$ of $G$ such that the following holds. Let $\hat{\G}$ be the associated relative Cayley graph. Then there is a constant $B$ such that for any $x \in G$, any $\eta \in \partial \hat{\G}$, and each $i \in \N$, we have $$\vert \{\gamma(i) : \gamma \in \text{CGR}(x, \eta)\} \vert \leq B$$ 
		
	\end{thm}
	
	\begin{proof}
		
		Take the finite generating set $X$ to be as in Theorem \ref{3.0}. Let $i \in \mathbb{N}$. Let $\nu$ be the constant from Theorem \ref{3.0}. Note that $\hat{\G}$ is $\nu$-hyperbolic. Fix any $\gamma_0 \in CGR(x,\eta)$ and let $k = i + 3\nu + 1$. We will show that for each $\gamma \in \text{CGR}(x, \eta)$, there exists a vertex $v$ on $\gamma_0$ with $d(v, \gamma_0(i)) \leq 3 \nu$ and such that $d_X(\gamma(i), v) \leq \nu$. 
		
		\vs
		
		Let $\gamma \in CGR(x,\eta)$ be arbitrary. Begin by joining $\gamma(k)$ and $\gamma_0(k)$ with a geodesic $\alpha$ (see Figure 1). By $\nu$-hyperbolicity of $\hat{\G}$, we have that $d(\gamma(k), \gamma_0(k)) \leq 2 \nu$, so $\alpha$ has length $\ell(\alpha)$ at most $2 \nu$.
		
		\vs
		
		Letting $\vert_k$ denote the restriction of a geodesic to $\{0,1,...,k\}$, we apply Theorem \ref{3.0} to the geodesic triangle with sides $\gamma_0 \vert_k, \alpha$ and $\gamma \vert_k$,  By Theorem \ref{3.0}, there exists a vertex $v$ on $\gamma_0 \vert_k$ or on $\alpha$ such that $d_X(\gamma(i), v) \leq \nu$. We cannot have $v$ on $\alpha$ because then we would have $d(\gamma(i), v) \leq \nu$ (since $d \leq d_X$), which would imply by the triangle inequality that $$k-i = d(\gamma(i), \gamma(k)) \leq d(\gamma(i), v) + d(v, \gamma(k)) \leq  d(\gamma(i), v) +\ell(\alpha) \leq \nu + 2 \nu = 3 \nu$$ contradicting our choice of $k$. Therefore, we must have that $v$ is on $\gamma_0 \vert_k$.

	\vs

Lastly, let us show that $d(v, \gamma_0(i)) \leq 3 \nu$. By $\nu$-hyperbolicity, we have $d(\gamma(i), \gamma_0(i)) \leq 2 \nu$, and note that $d_X(\gamma(i), v) \leq \nu$ implies $d(\gamma(i), v) \leq \nu$, so by the triangle inequality, $$d(v, \gamma_0(i)) \leq d(v, \gamma(i)) + d(\gamma(i), \gamma_0(i)) \leq \nu + 2 \nu = 3 \nu$$

\vs 

We conclude that for each $i \in \N$ and each $\gamma \in \text{CGR}(x, \eta)$, $\gamma(i)$ must be $\nu$-close in $d_X$ to a vertex $v$ on $\gamma_0$ with $d(v, \gamma_0(i)) \leq 3 \nu$. There are at most $6 \nu + 1$ such vertices on $\gamma_0$, so we obtain that $\vert \{\gamma(i) : \gamma \in \text{CGR}(x, \eta)\} \vert \leq (6 \nu+1) \vert B_X^{\nu}(1) \vert$. Thus, we set $B = (6 \nu + 1)\vert B_X^{\nu}(1)\vert$. 

\begin{figure}[H]
	\centering
	\includegraphics[width=0.5\linewidth]{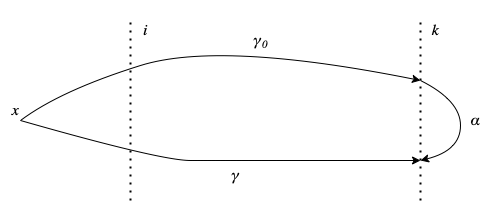}
	\caption{The arrangement of geodesics in the proof of Theorem \ref{3.1}.}
	\label{fig:33-arrangement}
\end{figure}

\end{proof}

As a corollary of Theorem \ref{3.1}, we obtain the following: 

\begin{thm}
	\label{3.4}
		Let $G$ be a finitely generated group hyperbolic relative to a collection of subgroups $\{H_1,...,H_n\}$. There exists a finite generating set $X$ of $G$ such that the following holds. If $\hat{\G}$ is the associated relative Cayley graph, then $\text{Geo}_1(x, \eta) \triangle \text{Geo}_1(y, \eta)$ is finite for each $x,y \in G$ and each $\eta \in \partial \hat{\G}$. 
\end{thm}

\begin{proof}
	Let $X$ be as in Theorem \ref{3.1}. By Theorem \ref{3.1}, we have that $\text{Geo}(x,\eta)$ is uniformly locally finite for each $x \in G$ and $\eta \in \partial \hat{\G}$. In \cite[Theorem 5.9]{MS20}, it is proved that if a hyperbolic graph $\G$ has the property that $\text{Geo}(x,\eta)$ is uniformly locally finite for each vertex $x$ and each $\eta \in \partial \G$, then $\text{Geo}_1(x, \eta) \triangle \text{Geo}_1(y, \eta)$ is finite for each pair of vertices $x,y$ and each $\eta \in \partial \G$. Therefore, $\text{Geo}_1(x, \eta) \triangle \text{Geo}_1(y, \eta)$ is finite for each $x,y \in G$ and each $\eta \in \partial \hat{\G}$. 
\end{proof}

\begin{rem}
	Note that Theorem \ref{3.1} implies that if a relatively hyperbolic group $G$ is generated by a finite set $X$ as in Theorem \ref{3.4}, and  if the set of ends of the associated relative Cayley graph is the same as $\partial \hat{\G}$, then the ends of $\hat{\G}$ have uniformly bounded degree (see \cite[Section 2]{Stall} for the definition of ends and the degree of an end). This appears to not have been known for relative Cayley graphs of relatively hyperbolic groups. 
\end{rem}

\section{Hyperfiniteness of the boundary action}
\label{4}

In this section, we establish the hyperfiniteness of the boundary actions of relatively hyperbolic groups as a consequence of Theorem \ref{3.4}. Our arguments follow \cite[Section 6]{MS20}. The main difference here is in our coding of labels of geodesics. In this section, we fix a finite generating set $X$ for $G$ as in Theorem \ref{3.4} and let $\hat{\G}$ denote the associated relative Cayley graph of $G$ with respect to $\{H_1,...,H_n\}$ and $X$. 

%We shall work in the general context of connected locally countable graphs. 

\vs 

%Let $X$ be a locally countable graph on a countable vertex set $V$. 

First, we give a binary coding to the symmetrized generating set $S :=( X \cup \mathcal{H})^{\pm}$. Using that $S$ is countably infinite, we fix a bijection $f: S  \to 2^{<\N} := \bigcup_{n \in \N}2^n$ from $S$ to the set $2^{<\N}$ of all finite binary sequences (which we can identify with the set of all finitely supported, infinite binary strings). The label of a geodesic ray is then coded as an element of $(2^{<\N})^{\N}$, the set of all infinite sequences of finite binary strings. 

\vs 

We will need to order elements of $(2^n)^n$ (i.e. the set of length $n$ sequences of length $n$ binary strings) for each $n$. Following \cite[Section 7]{DJK94}, for each $m_1, m_2 \in \N$, each $w = (w_0,w_1,...,w_{n-1}) \in (2^{m_1})^{m_2}$ and for each $n \in \N$ with $n \leq m_1, m_2$, we put $w \vert_n = ((w_0)\vert_n, (w_1) \vert_n,..., (w_{n-1}) \vert_n)$, where $(w_j)\vert_n$ is the restriction of the length $m_1$ binary sequence $w_j$ to the first $n$ entries. Similarly, if $w \in (2^{\N})^{\N}$, we put $w \vert_n =  ((w_0)\vert_n, (w_1) \vert_n,..., (w_{n-1}) \vert_n)$. If we visualize $w \in (2^n)^n$ as an $n \times n$ matrix, then $w \vert_i$ is an $i \times i$ submatrix of the $n \times n$ matrix $w$, starting at the top left corner of $w$. 

\vs 

For each $n \in \mathbb{N}$, we fix a total order $<_n$ on $(2^n)^n$ as in \cite[Section 7]{DJK94} such that for all $w,v \in (2^{n+1})^{n+1}$, $w\vert_n <_n v\vert_n \implies w <_{n+1} v$. Given $\gamma \in \text{CGR}(g, \eta)$, we define $\text{lab}(\gamma) \in (2^{<\N})^{\N}$ to be its coded label. Therefore, according to above, $\text{lab}(\gamma) \vert_n \in (2^n)^n$ denotes the restricted label. Now, analogously to \cite[Definition 6.1]{MS20}, we put:

\begin{defn}
	For $\eta \in \partial\hat{\G}$, define:
	\begin{equation*}
		C^{\eta} = \{(g, \text{lab}(\gamma) \vert_n) \in G \times (2^n)^n : g \in Geo_1(e, \eta), \gamma \in CGR(g, \eta), n \in \N\}
	\end{equation*}
\end{defn}

\begin{defn}
	An $s$ in $(2^n)^n$ \textbf{occurs} in $C^\eta$ if $(g, s) \in C^\eta$ for some $g \in Geo_1(e, \eta)$. An $s$ in $(2^n)^n$ \textbf{occurs infinitely often} in $C^\eta$ if $(g, s) \in C^\eta$ for infinitely many $g \in Geo_1(e, \eta)$. 
\end{defn}

Note that for each $n \in \N$, there exists $s \in (2^n)^n$ which occurs infinitely often in $C^\eta$ because taking any $\gamma \in \text{CGR}(e, \eta)$, by \cite[Proposition 5.8]{MS20}, $\gamma \setminus \text{Geo}_1(e, \eta)$ is finite, so there exists some $N$ such that for all $k \geq N$, $\gamma(k)\in \text{Geo}_1(e, \eta)$. Then $(\gamma(k), \text{lab}((\gamma(i))_{i \geq k})\vert_n) \in C^\eta$ and $\text{lab}((\gamma(i))_{i \geq k})\vert_n \in (2^n)^n$ for each $k \geq N$. Since $(2^n)^n$ is finite, by the Pigeonhole Principle, some $s \in (2^n)^n$ must repeat infinitely often in $C^{\eta}$, that is, $(\gamma(k), s) \in C^\eta$ for infinitely many $k \geq N$. For each $n \in \mathbb{N}$, we can therefore choose the minimal (in the order $<_n$ defined above) such $s \in (2^n)^n$ occuring infinitely often in $C^\eta$. We shall denote this element  by $s_n^{\eta}$.  

\vs 

\begin{prop}
	\label{4.res}
	For each $n \in \mathbb{N}$, we have that $(s_{n+1}^\eta) \vert_n = s_n^{\eta}$.
\end{prop}

\begin{proof}
	Since $s_{n+1}^{\eta}$ appears infinitely often in $C^\eta$, so does $(s_{n+1}^\eta) \vert_n $, so $s_n^\eta <_n (s_{n+1}^\eta) \vert_n $ or $s_n^\eta = (s_{n+1}^\eta) \vert_n$. If $s_n^\eta <_n (s_{n+1}^\eta) \vert_n $, then since there are only finitely many extensions of $s_n^\eta$ to an element of $(2^{n+1})^{n+1}$ and since $s_n^\eta$ appears infinitely often in $C^\eta$, there would exist $s \in (2^{n+1})^{n+1}$ such that $s \vert_n = s_n^{\eta}$ and $s$ appears infinitely often in $C^{\eta}$. Since $s\vert_n <_n (s_{n+1}^\eta) \vert_n $, we obtain that $s <_{n+1} s_{n+1}^\eta$, contradicting the minimality of $s_{n+1}^\eta$. Therefore, $s_n^\eta = (s_{n+1}^\eta) \vert_n$.
\end{proof}

\vs 

We now fix a total order $\leq $ on the group $G$ such that $g \leq h \implies d(e, g) \leq d(e, h)$ (for instance, fixing a total order on $S$, we can define $\leq$ to be lexicographic order on elements of $G$ as words over $S$, where we choose for each element of $G$ the lexicographically least word over $S$ representing it). Using the same notation as in \cite[Section 6]{MS20}, we put:

\begin{defn}
	For each $n \in \N$ and $\eta \in \partial \hat{\G}$, put $T_n^\eta = \{g \in Geo_1(e, \eta) : (g, s_n^\eta) \in C^\eta \}$ and put $g_n^\eta = \min T_n^\eta$ (where the minimum is with respect to the above total order on $G$). Put $k_n^\eta = d(e, g_n^\eta)$ for each $n \in \N$.
\end{defn}

Note that $\min T_n^\eta$ exists because $T_n^{\eta} \subseteq \text{Geo}(e,\eta)$ and $\text{Geo}(e,\eta)$ is locally finite by Theorem \ref{3.1}. By definition of $\leq$ and since $s_n^\eta = (s_{n+1}^{\eta})\vert_n$ for each $n$, we have that $(T_n^{\eta})_n$ is a non-increasing sequence of sets and therefore the sequence $(k_n^\eta)_{n \in \mathbb{N}}$ is a non-decreasing sequence of natural numbers. 

\vs 

We shall now generalize the results of \cite[Section 6]{MS20}, which were stated for Cayley graphs of hyperbolic groups. Recall that we fix our hyperbolicity constant to be $\nu$ from Theorem \ref{3.0}. We recall that the topology on $G$ is the discrete topology induced by the relative metric $d$, the topology on $\partial \hat{\G}$ is the canonical topology on the geodesic boundary, having countable neighbourhood base $V(\eta,m)^g = \{\mu \in \partial \hat{\G} : \exists \gamma \in \text{CGR}(g, \mu) \text{ and } \lambda \in \text{CGR}(g, \eta) \text{ with } d(\gamma(t), \lambda(t)) \leq 2 \nu \text{ for each } t \leq m\}$ for each $m \in \N$, each $\eta \in \partial \hat{\G}$ and each basepoint $g \in G$, $G^{\mathbb{N}}$ has the product topology and $C_{hb}(\hat{\G})$ has the topology of pointwise convergence. 

\vs

Let us establish a link between the topology of  $\partial \hat{\G}$ and sequences of CGRs in $\hat{\G}$. The condition in the following proposition is often used as the definition of the topology on $\partial X$ when $X$ is a proper hyperbolic space, but in general does not give the same topology on $\partial X$ that we work with here. 

\begin{prop}
	\label{4.1}
	Suppose that $\eta_n \to \eta$ in $\partial \hat{\G}$. Then for any $g \in G$, there exists a sequence of CGRs $(\gamma_n)_n$ such that $\gamma_n \in CGR(g, \eta_n)$ for each $n$ and such that every subsequence of $(\gamma_n)_n$ itself has a subsequence which converges to some CGR $\gamma \in CGR(g, \eta)$. 
\end{prop}

\begin{proof}
	Since $\eta_n \to \eta$, by definition of the topology on $\partial \hat{\G}$, we have that for each $m \in \N$, there exists a CGR $\gamma_m \in \text{CGR}(g, \eta_m)$ and $\lambda_m \in \text{CGR}(g, \eta)$ such that $d(\gamma_m(t), \lambda_m(t)) \leq 2 \nu$ for every $t \leq m$. Fixing any $\lambda \in \text{CGR}(g, \eta)$, we obtain that $d(\gamma_m(t), \lambda(t)) \leq 4 \nu$ for every $t \leq m$ and every $m$, since $\lambda_m, \lambda \in \text{CGR}(g, \eta)$ for all $m$ and hence are $2 \nu$ close for each $m$. We claim that every subsequence of $(\gamma_n)_n$ has a convergent subsequence. First, let us argue as in the proof of Theorem \ref{3.1} to show that for each $i$, $\vert \{\gamma_n(i) : n \in \N\}\vert $ is finite.
	
	\vs 
	
	Given $i \in \N$, set $k = i + 5\nu + 1$. For each $n \geq k$, we have $d(\gamma_n(k), \lambda(k)) \leq 4 \nu$. Let $u$ denote a geodesic between  $\gamma_n(k)$ and $\lambda(k)$ (see Figure 2). 
	
	\begin{figure}[H]
		\centering
		\includegraphics[width=0.5\linewidth]{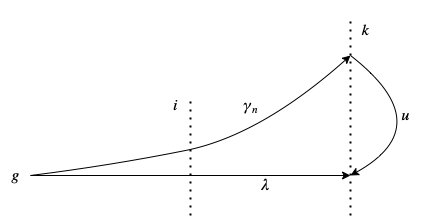}
		\caption{The geometry of the geodesics $\gamma_n$, $\lambda$.}
		\label{fig:top-lemma}
	\end{figure}
	
	Then arguing as in the proof of Theorem \ref{3.1}, there exists a vertex $v$ on $\lambda$ with $d_X(\gamma_n(i), v) \leq \nu$. It follows that $\vert \{\gamma_n(i) : n \geq k\} \vert$ is finite, and therefore that $\vert \{\gamma_n(i) : n \in \N\} \vert$ is finite. Therefore, $\bigcup_n \gamma_n \cup \lambda$ is locally finite. Since $\bigcup_n \gamma_n \cup \lambda$ is locally finite, by Kőnig's lemma it follows that every subsequence of $(\gamma_n)_n$ has a convergent subsequence. The limit CGR $\gamma$ of this subsequence is in $\text{CGR}(g,\eta)$ because for each $t$, $d(\gamma_k(t), \lambda(t)) \leq 4 \nu$ for all but finitely many $k$, so $d(\gamma(t), \lambda(t)) \leq 4 \nu$ for all $t$.

\end{proof}

%The following is a generalization of Claim 6.2 in \cite{MS20}. In \cite{MS20}, the set $C$ below is shown to be compact, however here it is only closed. 
	We now generalize the claims of \cite[Section 6]{MS20} to relatively hyperbolic groups. We begin by generalizing Claim 1 of \cite{MS20}. In Claim 1 in \cite{MS20}, the set $C$ below is proved to be compact, while here it is only closed. 

\begin{clm}
	\label{4.2}
	The set $C = \{\gamma \in G^{\mathbb{N}} : \gamma \text{ is a CGR}\}$ is closed. Furthermore, for any $g \in G$ and any $\eta \in \partial \hat{\G}$, the set $CGR(g, \eta) \subseteq G^{\N}$ is compact. 
\end{clm}

\begin{proof}
	
	Let $(\gamma_n)_n$ be a sequence of elements of $C$ converging  to some $\gamma \in G^{\mathbb{N}}$. We claim that $\gamma$ is a geodesic. Indeed, since $\gamma_n \to \gamma$, for each $m \in \mathbb{N}$, there exists $N \in \mathbb{N}$ such that for all $n \geq N$, we have $\gamma_n \vert_m  = \gamma \vert_m $.  In particular, it follows that $\gamma \vert_m $ is a geodesic, since $\gamma_n \vert_m$ is a geodesic for each $n$. Thus, $\gamma$ is a geodesic ray based at $\lim_n \gamma_n(0)$ and is hence a CGR, so $\gamma \in C$. Therefore, $C$ is closed. 
	
	\vs
	
	The "furthermore" statement follows immediately from Kőnig's lemma, since $\text{Geo}(g, \eta)$ is locally finite (by Theorem \ref{3.1}). 
	
\end{proof}

The next claims are the exact relatively hyperbolic analogues of claims from \cite{MS20} and their proofs are almost identical (most proofs are completely identical), however, we present all proofs for completeness.

\begin{clm} 
	\label{4.3}
	The set $R = \{(\eta, g, \gamma) \in \partial \hat{\G} \times  G^{\mathbb{N}} : \gamma \in CGR(g, \eta)\}$ is closed in $\partial \hat{\G} \times G^{\mathbb{N}}$. 
\end{clm}

\begin{proof}
	
	Suppose that $(\eta_n , g_n, \gamma_n) \in R$ for all $n$ and that $(\eta_n, g_n, \gamma_n) \to (\eta, g, \gamma)$. Then $\eta_n \to \eta \in \partial \hat{\G}$, $g_n \to g$ in $G$ (so that $(g_n)$ is eventually equal to $g$, by discreteness of $G$) and $\gamma_n \to \gamma$ in $G^{\mathbb{N}}$, so that $\gamma \in\text{CGR}(g, \eta')$ for some $\eta' \in \partial \hat{\G}$ (by Claim \ref{4.2}). We will show that $\eta = \eta'$. 
	
	\vs 
	
	As $\eta_n \to \eta$, by Proposition \ref{4.1}, there exists a sequence $(\gamma'_n)_n$ with $\gamma'_n \in \text{CGR}(g, \eta_n)$ which has a subsequence $(\gamma_{n_k}')_k$ that converges to some $\gamma' \in CGR(g, \eta)$. Choose $k$ large enough such that all $g_{n_k}$ equal $g$, so that $\gamma_{n_k}, \gamma'_{n_k}\in CGR(g, \eta_{n_k})$. We then have that $d(\gamma_{n_k}(m), \gamma'_{n_k}(m)) \leq 2 \nu$ for each $m$. Taking $k \to \infty$, we obtain that $d(\gamma(m), \gamma'(m)) \leq 2 \nu$ for all $m$, and therefore that $\eta = \eta'$. Thus, $(\eta_n, g_n, \gamma_n) \to (\eta, g, \gamma)$ with $\gamma \in \text{CGR}(g, \eta)$, so $(\eta, g, \gamma) \in R$ and so $R$ is closed. 
	
\end{proof}

\begin{clm} 
	\label{4.4}
	The set $F = \{(\eta, g, (\gamma(0), \gamma(1)...,\gamma(n))) \in \partial \hat{\G}\times  G \times G^{< \mathbb{N}}: \gamma \in CGR(g, \eta)\}$ is Borel in $\partial \hat{\G} \times G \times G^{< \mathbb{N}}$. 
\end{clm}

\begin{proof}
	Let $F' = \{(\eta, g, (\gamma(0), \gamma(1),...,\gamma(n) ),  \gamma' )\in \partial \hat{\G} \times  G \times G^{< \mathbb{N}} \times G^{\mathbb{N}}: (\eta, g, \gamma') \in R \text{ and } \gamma'_i = \gamma_i \text{ for each } 0 \leq i \leq n\}$. By Claim \ref{4.3}, $F'$ is closed in $\partial \hat{\G} \times  G \times G^{< \mathbb{N}} \times G^{\mathbb{N}}$. Note that $F$ is the projection of $F'$ to the first 3 components $\partial \hat{\G} \times G \times G^{<\N}$. Note also that the section $F'_{(\eta, g, (\gamma(0), \gamma(1),...,\gamma(n) ))}$ is compact for every $(\eta, g, (\gamma(0), \gamma(1),...,\gamma(n) )) \in \partial \hat{\G} \times  G^{< \mathbb{N}}$. Indeed,  $F'_{(\eta, g, (\gamma(0), \gamma(1),...,\gamma(n) ))} = \{\gamma' \in CGR(g, \eta) : \gamma'(i)= \gamma(i) \text{ for all } 0 \leq i \leq n\}$, which is a closed subset of the compact set $\text{CGR}(g, \eta)$, hence it is compact. By \cite[Theorem 18.18]{Kech95}, it follows that $F$ is Borel in $\partial \hat{\G} \times  G \times G^{< \mathbb{N}}$. 
\end{proof}

\begin{clm} 
	\label{4.5}
	The set $M = \{(\eta, \xi) \in \partial \hat{\G} \times C_{hb}(\hat{\G}): \xi \in \Xi (\eta)\}$ is Borel in $\partial \hat{\G} \times C_{hb}(\hat{\G})$. 
\end{clm}

\begin{proof}
	We follow a similar proof to the proof of \cite[Claim 4]{MS20}. We will show that $M$ is both analytic and coanalytic, hence Borel by \cite[Theorem 14.11]{Kech95}. By definition of $\Xi(\eta)$, we have that $(\eta, \xi) \in M$ if and only if
	
	\begin{equation*}
		\exists \gamma \in G^{\N} : (\eta, \gamma(0), \gamma) \in R \text{ and } \xi_{\gamma} = \xi
	\end{equation*}

	We also have that 
	
	\begin{equation*}
		\xi_{\gamma} = \xi \iff \forall g \in G\text{ }\exists n \in \N \text{ }\forall m \geq n \text{ }f_{\gamma(m)}(g) = \xi(g)
	\end{equation*}
	
	which gives a Borel definition of the set $\{(\xi, \gamma) \in C_{hb}(\hat{\G}) \times C : \xi_{\gamma} = \xi \}$. Thus, from Claim \ref{4.3} we have that $M$ is analytic. To show that $M$ is coanalytic, we will show the following, denoting $N_k(A)$ the $k$-neighbourhood of a subset $A$ of $G$:
	
	\begin{align*}
		&(\eta, \xi) \in M \text{ if and only if } \forall \lambda \in G^{\mathbb{N}} \text{ if } (\eta, e, \lambda) \in R, \text{ then } \forall k \in \mathbb{N}, \exists \gamma^k \in G^{k+1}\text{ a geodesic path with  } \\
		&\gamma^k(0) = e \text{ such that } \gamma^{k}\subseteq N_{2 \nu} (\lambda) \text{ and such that } \forall g \in G, \exists n_g \in \mathbb{N} \text{ such that } \forall i,j > n_g, f_{\gamma^j(i)}(g) = \xi(g)
	\end{align*}
	
	This formula defines a coanalytic set since there is a single universal quantifier $\forall$ ranging over an uncountable standard Borel space $G^{\N}$. 
	\vs
	
	For the forward direction, if $(\eta, \xi) \in M$, then there exists $\gamma \in \text{CGR}(e, \eta)$ converging to $\xi$. We simply take $\gamma^k = \gamma \vert_{k}$ (the restriction of $\gamma$ from 0 to $k$) for each $k \in \mathbb{N}$. Then for each $\lambda \in \text{CGR}(e, \eta)$, we have $d(\gamma(n), \lambda(n)) \leq 2 \nu$ for each $n \in \mathbb{N}$, so $\gamma^k \subseteq N_{2 \nu}(\lambda)$ for each $k$. Furthermore, since $\gamma$ converges to $\xi$, we have that for all $\forall g \in G$, there exists $n_g$ such that for all $i,j > n_g$, we have $f_{\gamma^j(i)}(g) = \xi(g)$. 
	
	\vs 
	
	For the reverse direction, let $\lambda \in \text{CGR}(e, \eta)$. Then there exists a sequence $\gamma^k \in G^{k+1}$ of geodesic paths starting at $e$, each contained in $N_{2 \nu}(\lambda)$ and such that $f_{\gamma^i(j)}(g) \to \xi(g)$. For each $i$, fix $k = i + 3 \nu + 1$ and using $\nu$-hyperbolicity, choose an $N$ sufficiently large such that for all $n \geq N$, we have $$d(\gamma_n(t), \lambda(t)) \leq 2 \nu$$ for all $t \leq k$. Arguing as in the proof of Theorem \ref{3.1}, we have that $\{\gamma^j(i): j \geq N\}$ is finite, so that $\{\gamma^j(i) : j \in \N\}$ is finite for each $i$. Therefore, by Kőnig's lemma, $(\gamma^k)_k$ has a subsequence converging to some CGR $\gamma$ based at $e$, and $\gamma \subseteq N_{2 \nu}(\lambda)$, so $\gamma \in \text{CGR}(e, \eta)$. From $f_{\gamma^i(j)}(g) \to \xi(g)$ as $i,j \to \infty$, we have that $\xi_{\gamma} = \xi$.  Since $\gamma \in \text{CGR}(e, \eta)$, we conclude that $(\eta, \xi) \in M$. 
	
\end{proof}

By \cite[Proposition 5.2]{MS20}, for each $\eta \in \partial \hat{\G}$, we have that the section $M_{\eta} =\Xi(\eta)$ is finite, having cardinality bounded above by the constant $B$ from Theorem \ref{3.1}. Since $M$ is Borel and has finite sections of size at most $B$, by the Lusin-Novikov theorem we have Borel functions $\xi_1,..., \xi_B : \partial \hat{\G} \to C_{hb}(\hat{\G})$ such that $M$ is the union of the graphs $G_{\xi_i} = \{(\eta, \xi_i(\eta)) : \eta \in \partial \hat{\G}\}$ of the $\xi_i$. 

\begin{clm}
	\label{4.6}
	
	For each $i =1,...,B$, $Q_i = \{(\eta, g, h) \in \partial \hat{\G} \times G^2 : h \in Q(g, \xi_i(\eta))\}$ is Borel in $\partial \hat{\G} \times G^2$. 
	
\end{clm}

\begin{proof}
	
	By \cite[Lemma 4.2]{MS20}, for $x,y \in G$, denoting $\gamma(x,y)$ the union of all geodesic paths in $\hat{\G}$ from $x$ to $y$, we have that $Q(g, \xi_i(\eta)) = \bigcup_{n \in \N} \gamma(g, x_n)$ for some, equivalently any, $(x_n)_n \in \text{CGR}(g, \eta)$ converging to $\xi_i(\eta)$. From this, we obtain that:
	
	\begin{equation*}
		h \in Q(g, \xi_i(\eta)) \iff \exists \lambda \in C \text{(resp. $\forall \lambda \in C$) : $\lambda(0) = g$ and $\xi_{\lambda} = \xi_i(\eta)$ and $\exists n \in \N$ : $h \in \gamma(g, \lambda(n)) $}
	\end{equation*}

	This yields the analyticity (from the $\exists$ above) and coanalyticity (from the $\forall$ above) of $Q_i$, hence Borelness of $Q_i$.
	
\end{proof}

\begin{clm}
	\label{4.7}
	
	The set $P = \{(\eta, h) \in \partial \hat{\G} \times G : h \in \hat{\G}_{s, \eta}\}$ is Borel in $\partial \hat{\G} \times G$. 
\end{clm}

\begin{proof}
	
	We have that $h \in \hat{\G}_{s, \eta}$ if and only if: 
	
	\begin{equation*}
		\forall n \in \N, \exists \gamma^n \in G^{n+1} : (\eta, h, \gamma^n) \in F \text{ and } \forall i \leq B \text{ }\forall k < n, (\eta, h, \gamma^n(k)) \in Q_i
	\end{equation*}
	
	Indeed, if $h \in \hat{\G}_{s, \eta}$, then $\bigcap_{\xi \in \Xi(\eta)} Q(h, \xi)$ contains a CGR $\gamma \in \text{CGR}(h, \eta)$, so we can take $\gamma^n = \gamma \vert_{n}$ (the restriction from 0 to $n$) for all $n \in \N$ to satisfy the above condition.
	
	\vs 
	
	Conversely, if the above condition holds, then by local finiteness of $\text{Geo}(h, \eta)$, the sequence $(\gamma^n)_{n \in \N}$ with $(\eta, h, \gamma^n) \in F$ will have a subsequence converging to some $\gamma \in \text{CGR}(h, \eta)$ and the above condition yields that $\gamma \subseteq \bigcap_{\xi \in \Xi(\eta)} Q(h, \xi)$, so that $h \in \hat{\G}_{s, \eta}$. 
	
	\vs 
	
	Since $F$ and $Q_i$ are Borel, we conclude that $P$ is Borel.
	
\end{proof}

\begin{clm}
	\label{4.8}
	
	The set $P_1 = \{(\xi, \eta, h) \in C_{hb}(\hat{\G}) \times \partial \hat{\G} \times G : h \in \hat{\G}_{s, \eta} \text{ and } \xi = \xi_{h, \eta}\}$ is Borel in $C_{hb}(\hat{\G}) \times \partial \hat{\G} \times G$. 
\end{clm}

\begin{proof}
	
	We have $(\xi, \eta, h) \in P_1$ if and only if: 
	
	\begin{equation*}
		(\eta, h) \in P \text{ and $\exists i \leq B$ : $(\eta, \xi) \in G_{\xi_i}$ and $\forall j \leq B$, $Q(h, \xi_i(\eta)) \subseteq Q(h, \xi_j(\eta)) $}
	\end{equation*}

	Since $P$ is Borel (Claim \ref{4.7}), $G_{\xi_i}$ is Borel (as $\xi_i$ is Borel), and $Q_i$ is Borel (Claim \ref{4.6}), the above yields that $P_1$ is Borel.
	
\end{proof}

\begin{clm}
	\label{4.9}
	
	The set $L = \{(h, \xi, \eta) \in G \times C_{hb}(\hat{\G}) \times \partial \hat{\G} : h \in Y(e, \xi), \xi \in \Xi(\eta)\}$ is Borel in $G \times C_{hb}(\hat{\G}) \times \partial \hat{\G}$. 
\end{clm}

\begin{proof}
	
	We have that $(h, \xi, \eta) \in L$ if and only if $(\eta, \xi) \in M $ and $h$ is the closest element to $e$ (in the metric $d$) such that $h \in \text{Geo}(e, \eta)$ and $(\xi, \eta, h) \in P_1 $. Thus, by Claims \ref{4.5}, \ref{4.6}, \ref{4.8}, $L$ is Borel (note that $h \in \text{Geo}(e, \eta) \iff (\eta, e, h) \in Q_i$ for some $i \leq B$, so $\{(h,\eta) \in G \times \partial \hat{\G} : h \in \text{Geo}(e, \eta)\}$ is Borel in $G \times \partial \hat{\G}$ by Claim \ref{4.6}). 
	
\end{proof}

\begin{clm}
	\label{4.10}
	
	The set $B = \{(g, h, \xi, \eta) \in G^2 \times C_{hb}(\hat{\G}) \times \partial \hat{\G} : g \in Q(h, \xi), h \in Y(e, \xi), \xi \in \Xi(\eta)\}$ is Borel in $G^2 \times C_{hb}(\hat{\G}) \times \partial \hat{\G}$.
\end{clm}

\begin{proof}
	
	We have that $(g,h,\xi,\eta) \in B$ if and only if $\exists i \leq r$ such that $\xi = \xi_i(\eta)$ and $(\eta, h, g) \in Q_i$ and $(h, \xi, \eta) \in L$. Since $L, Q_i$ and $\xi_i$ are Borel, it follows that $B$ is Borel.
	
\end{proof}

\begin{clm} 
	\label{4.11}
	
	The set $A = \{(\eta, g) \in \partial \hat{\G} \times G : g \in \text{Geo}_1(e, \eta)\}$ is Borel in $\partial \hat{\G} \times G$. 
	
\end{clm}

\begin{proof}
	
	Since $\text{Geo}_1(e, \eta) = \bigcup_{\xi \in \Xi(\eta)} \bigcup_{h \in Y(e, \xi)} Q(h, \xi)$, we have:
	
	\begin{equation*}
		(\eta, g) \in A \iff \exists \xi \in \Xi(\eta), \exists h \in Y(e, \xi) : g \in Q(h, \xi) \iff \exists \xi \in \Xi(\eta), \exists h \in Y(e, \xi) : (g,h,\xi, \eta) \in B
	\end{equation*}

 Therefore, $A$ is the projection $(g,h,\xi,\eta) \mapsto (\eta,g)$ of $B$ onto $\partial \hat{\G} \times G$. By Claim \ref{4.10}, $B$ is Borel. Also, the sections $\{(h, \xi) \in G \times C_{hb}(\hat{\G}) : h \in Y(e, \xi), \xi \in \Xi(\eta)\}$ of $B$ are finite by Theorem \ref{3.1} and \cite[Proposition 5.2]{MS20}. Therefore, by the Lusin-Novikov theorem, $A$ is Borel. 
	
\end{proof}

\begin{clm}
	\label{4.12}
	
	The set $D = \{(\eta, (\gamma(0), \gamma(1),...,\gamma(n) )) \in \partial \hat{\G}\times G^{< \mathbb{N}} : \gamma(0) \in Geo_1(e, \eta) \text{ and } \gamma \in CGR(\gamma(0), \eta)\}$ is Borel in $\partial \hat{\G} \times G^{< \mathbb{N}}$. 
	
\end{clm}

\begin{proof}
	
	We have that $(\eta, (\gamma(0), \gamma(1),...,\gamma(n) )) \in D$ if and only if $(\eta, \gamma(0), (\gamma(0), \gamma(1),...,\gamma(n) )) \in F$ and $(\eta, \gamma(0)) \in A$. By Claim \ref{4.11}, $A$ is Borel in $\partial \hat{\G} \times G$. Also, $F$ is Borel by Claim \ref{4.4}. Therefore, $D$ is Borel. 
	
\end{proof}

\begin{clm} 
	\label{4.13}
	For each $n$, the set $S_n := \{(\eta, s^n)  \in \partial \hat{\G} \times (2^n)^n : s^n = s_n^{\eta}\}$ is Borel in $\partial \hat{\G} \times (2^n)^n$. 
\end{clm}

\begin{proof}
	
	We have that $(\eta, s^n) \in S_n$ if and only if $s^n$ is the $<_n$-minimal element in $(2^n)^n$ for which the following holds:
	
	\begin{equation*}
		\forall m \in \mathbb{N}, \exists (\gamma(0), \gamma(1),...,\gamma(n) ) \in G^{n+1} : d(\gamma(0), e) \geq m, (\eta, (\gamma(0), \gamma(1),...,\gamma(n) )) \in D \text{ and $\text{lab}(\gamma)\vert_n= s^n $}
	\end{equation*}
	
	Note that the the "only if" holds by local finiteness of $\text{Geo}(e,\eta)$. Thus, $S_n$ is Borel by Claim \ref{4.12}. 
	
\end{proof}

Now let $E$ denote the orbit equivalence relation of the action of $G$ on $\partial \hat{\G}$. 

\begin{defn}
	\label{4.14}
	Let $Z = \{\eta \in \partial \hat{\G} : k_n^{\eta} \nrightarrow \infty\}$. 
\end{defn}

Since $\text{Geo}_1(e, \eta)$ is locally finite (as $\text{Geo}(e, \eta)$ is locally finite and $\text{Geo}_1(e, \eta) \subseteq \text{Geo}(e, \eta)$), we have that $Z$ is the set of all $\eta$ such that there exists $g_{\eta}$ belonging to $T_{n}^{\eta}$ for all $n$, i.e. for which there exists $\gamma^{\eta} \in CGR(g_{\eta}, \eta)$ with label $s^{\eta} \in (2^{< \N})^{\N}$. 

\begin{lem}
	\label{4.15}
	The map $\alpha:( Z, E \vert_Z) \to (\partial \hat{\G}, =) $ given by $\eta \mapsto g_{\eta}^{-1} \eta$ is a Borel reduction. 
\end{lem}

\begin{proof}
	We argue as in \cite{MS20}. First, let us show that $s_n^{\eta} = s_n^{g \eta}$ for each $g \in G$, each $\eta \in \partial \hat{\G}$ and each $n \in \N$. If there are infinitely many pairs $(h, s_n^{\eta}) \in C^{\eta}$, then since the left action of $G$ on $\hat{\G}$ preserves labels of geodesics, there are infinitely many pairs $(gh, s_n^{\eta})$, where $s_n^{\eta} = \text{lab}(\gamma)\vert_n$ for some $\gamma \in \text{CGR}(\gamma(0), g \eta)$ and where $\gamma(0) \in g \text{Geo}_1(e, \eta) = \text{Geo}_1(g, g \eta)$ (using \cite[Lemma 5.10]{MS20} in the last line). 
	
	\vs 
	
	By \cite[Theorem 5.9]{MS20}, the symmetric difference between $\text{Geo}_1(g, g \eta)$ and $\text{Geo}_1(e, g \eta)$ is finite and so there are infinitely many pairs $(gh, s_n^{\eta}) \in \text{Geo}_1(e, g \eta)$. Hence, there are infinitely many pairs $(gh, s_n^{\eta}) \in C^{g \eta}$. Thus, as $s_n^{\eta}$ is least in the order $<_n$ that appears infinitely often in $C^{\eta}$, we have that $s_n^{\eta} = s_n^{g \eta}$. As $s_n^{\eta} = s_n^{g \eta}$ for each $n$, we have $s^{\eta} = s^{g \eta}$. 
	
	\vs 
	
	This implies that $\alpha$ is constant on $G$-orbits. Indeed, suppose $\theta = g \eta$ for some $g \in G$, $\eta, \theta \in Z$. We have that $\alpha$ maps the boundary point $[\gamma^{\theta}]$ to the boundary point $[g_{\theta}^{-1} \gamma^{\theta}]$. Note that $g_{\theta}^{-1} \gamma^{\theta} \in \text{CGR}(e, g_{\theta}^{-1}\theta)$ and $\text{lab}(g_{\theta}^{-1}\gamma^{\theta}) = s^{\theta}$, because $\gamma^{\theta}$ has label $s^{\theta}$ and left multiplication preserves labels of geodesics. On the other hand, $\alpha$ maps $\eta = [\gamma^{\eta}]$ to $g_{\eta}^{-1} \eta = [g_{\eta}^{-1} \gamma^{\eta}]$. We have that $g_{\eta}^{-1} \gamma^{\eta} \in \text{CGR}(e, g_{\eta}^{-1}\eta)$ and $\text{lab}(g_{\eta}^{-1}\gamma^{\eta}) = s^{\eta}$. But by above, $s^{\eta} = s^{g \eta} = s^{\theta}$. Therefore, $g_{\eta}^{-1} \gamma^{\eta} $ and $g_{\theta}^{-1} \gamma^{\theta}$ both start at $e$ and have the same label. Therefore, they are the same geodesic. Hence, $g_{\theta}^{-1}\theta = g_{\eta}^{-1}\eta$ i.e. $\alpha(\theta) = \alpha(\eta)$.
	
	\vs 
	
	It follows that $\alpha$ is reduction to $=$ on $\partial \hat{\G}$. Indeed, the above shows that $\theta E \eta \implies \alpha(\theta) = \alpha(\eta)$. Conversely, if $\alpha(\theta) = \alpha(\eta)$, then $g_{\eta}^{-1} \eta = g_{\theta}^{-1} \theta $, so $\theta = g_{\theta}g_{\eta}^{-1} \eta $, and therefore $\theta E \eta$. 
	
	\vs 
	
	It remains to show that $\alpha$ is Borel. To show this, let us first  show that the set  $U := \{(\eta, s) \in Z \times (2^{\N})^{\N} : s = s^{\eta}\}$ is Borel. We have $s = s^{\eta}$ if and only if $(\eta, s \vert_n) \in S_n$ for each $n \in \N$, so $\{(\eta, s) \in Z \times (2^{\N})^{\N} : s = s^{\eta}\}$ is Borel by Claim \ref{4.13} (note that the map $(\eta, s) \mapsto (\eta, s\vert_n)$ is continuous, hence Borel, for each $n \in \N$). 
	
	\vs
	
	Now the Borelness of $U$ implies the Borelness of the graph of $\alpha$. Indeed, note that for $\eta \in Z$ and $\theta \in \partial \hat{\G}$, denoting $\text{lab} : Z \times C \to Z \times (2^{\N})^{\N}$ the continuous map $(\eta,\gamma) \mapsto (\eta, \text{lab}(\gamma))$, we have:
	
	\begin{align*}
		\theta = g_{\eta}^{-1} \eta &\iff \exists \gamma \in C : \gamma \in \text{CGR}(e, \theta) \text{ and } \text{lab}(\gamma) = s^{\eta} \\
		&\iff \exists \gamma \in C : (\theta, e, \G) \in R \text{ and } (\eta, \gamma) \in \text{lab}^{-1}(U) 
	\end{align*}
	
	Putting $T = \{(\eta, \theta, \gamma) \in Z \times \partial \hat{\G} \times C : (\theta, e, \gamma) \in R \text{ and } (\eta, \gamma) \in \text{lab}^{-1}(U)\}$, we have that $T$ is Borel because $R$ and $U$ are Borel (see Claim \ref{4.3} for the Borelness of $R$). By above, the graph of $\alpha$ is the projection $\text{proj}_{Z \times \partial \hat{\G}}(T)$ of $T$ onto the first two coordinates $(\eta, \theta)$. For each $(\eta, \theta) \in Z \times \partial \hat{\G}$, the section $T_{(\eta, \theta)} = \{\gamma \in C : (\eta, \theta, \gamma) \in T\} = \{\gamma \in C : \gamma \in \text{CGR}(e, \theta) \text{ and } \text{lab}(\gamma) = s^{\eta}\}$ is finite, being either a singleton or the empty set (because a geodesic ray is uniquely determined by its basepoint and label). Therefore, by the Lusin-Novikov theorem, we have that $\text{proj}_{Z \times \partial \hat{\G}}(T)$ is Borel. Thus, the graph of $\alpha$ is Borel, so $\alpha$ is Borel.
	
\end{proof}

\begin{lem}
	\label{4.16}
	$E$ is smooth on the saturation $[Z]_E = \{\eta \in \partial \hat{\G} : \exists \theta \in Z \text{ such that } \theta E \eta\}$.
\end{lem}

\begin{proof}
	By Lemma \ref{4.15}, $E$ is smooth on $Z$, yielding the claim. 
\end{proof}

\begin{defn}
	\label{4.17}
	Let $Y = \partial \hat{\G} \setminus [Z]_E$. For each $n \in \mathbb{N}$, define $H_n : \partial \hat{\G} \to 2^G$ by $H_n(\eta) = (g_{n}^{\eta})^{-1} T_n^{\eta}$.  Let $F_n$ be the equivalence relation on $\mathrm{im} H_n$ which is the restriction of the shift action of $G$ on $2^G$ to $\mathrm{im} H_n$.
\end{defn}

The following lemma is a generalization of \cite[Lemma 6.7]{MS20}.

\begin{lem}
	\label{4.18}
	There exists a constant $K$ such that for each $n \in \mathbb{N}$, each equivalence class of $F_n$ has size at most $K$. 
\end{lem}

\begin{proof}
	
	Note that by Thereom \ref{3.1}, we have that each closed ball of radius $r$ in $\text{Geo}(x, \eta)$ has cardinality at most $(2(r+2\nu)+1)B$, where $B$ is the constant from Theorem \ref{3.1}. We will show that we can take $K =(20 \nu + 1)  B$.
	
	\vs
	
	Let $\eta, \theta \in \partial \hat{\G}$ and suppose that $H_n(\eta) = g H_n(\theta)$. By the proof of \cite[Lemma 6.7]{MS20} (which only relies on the hyperbolicity of the Cayley graph and local finiteness of geodesic ray bundles and so holds in our context when applied to $\hat{\G}$), we have $d(e,g) \leq 8\nu$. For completeness, let us reproduce this proof. 
	
	\vs 
	
	By defiinition, $T_n^{\eta}$ (resp. $T_n^{\theta}$) is an infinite subset of $\text{Geo}(e, \eta)$ (resp. $\text{Geo}(e, \theta)$). Since $\text{Geo}(e, \eta)$ is locally finite, this means that $T_n^{\eta}$ (resp. $T_n^{\theta}$) uniquely determines $\eta$ (resp. $\theta$). From $H_n(\eta) = g H_n(\theta)$, we have $(g_n^{\eta})^{-1} T_n^{\eta} = g(g_n^{\theta})^{-1} T_n^{\theta}$ and since $T_n^{\eta}$ and $T_n^{\theta}$ determine their boundary points, this implies that $ (g_n^{\eta})^{-1} \eta = g(g_n^{\theta})^{-1} \theta$. Let us denote by $\sigma$ the common boundary point $ (g_n^{\eta})^{-1} \eta = g(g_n^{\theta})^{-1} \theta$.

	\begin{figure}[H]
		\centering
		\includegraphics[width=0.4\linewidth]{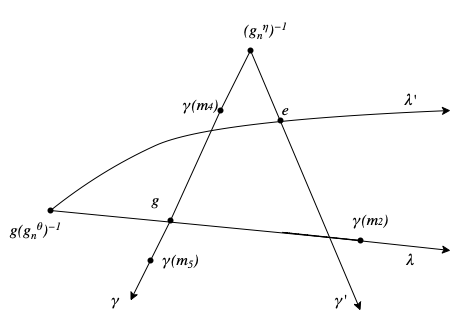}
		\caption{The geometry of the proof of Lemma \ref{4.18}.}
		\label{fig:lemma-418}
	\end{figure}
	
	We have that $g,e \in (g_n^{\eta})^{-1} T_n^{\eta} = g(g_n^{\theta})^{-1} T_n^{\theta} \subseteq \text{Geo}(g(g_n^{\theta})^{-1},\sigma)$, so there exists $\lambda \in \text{CGR}(g(g_n^{\theta})^{-1},\sigma)$ passing through $g$ and $\lambda' \in \text{CGR}(g(g_n^{\theta})^{-1},\sigma)$ passing through $e$. Write $g = \lambda(m_1)$ and $e = \lambda'(m_2)$ for some $m_1, m_2 \in \N$. Note that by $\nu$-hyperbolicity, we have $d(e, \lambda(m_2)) \leq 2 \nu$. Also, we have $m_2 \geq m_1$. Indeed, since $g_n^{\theta}g^{-1} \in g_n^{\theta} g^{-1}(g_n^{\eta})^{-1}T_n^{\eta} = T_n^{\theta}$, we have:

		$$m_2 = d(e,g(g_n^{\theta})^{-1}) = d(e, g_n^{\theta}g^{-1}) 
		\geq d(e, g_n^{\theta}) 
		= d(e,(g_n^{\theta})^{-1}) 
		= d(g, g(g_n^{\theta})^{-1}) 
		= m_1$$

	where $d(e, g_n^{\theta}g^{-1}) 
	\geq d(e, g_n^{\theta}) $ holds by $\leq$-minimality of $g_n^{\theta}$ in $ T_n^{\theta} $.
	\vs 
	
	Similarly, from $g,e \in (g_n^{\eta})^{-1} T_n^{\eta} $, we have $g,e \in \text{Geo}((g_n^{\eta})^{-1},\sigma)$, and so there exists $\gamma \in \text{CGR}((g_n^{\theta})^{-1},\sigma)$ passing through $g$ and $\gamma' \in \text{CGR}((g_n^{\theta})^{-1},\sigma)$ passing through $e$. Write $g = \gamma(m_3)$ and $e = \gamma'(m_4)$ for some $m_3, m_4 \in \N$ (see Figure 3). By $\nu$-hyperbolicity, we have $d(e, \gamma(m_4)) \leq 2 \nu$ and $m_4 \leq m_3$ because $g_n^{\eta} g \in g_n^{\eta}g(g_n^{\theta})^{-1}T_n^{\theta} = T_n^{\eta}$ and so by the $\leq$-minimality of $g_n^{\eta}$ in $T_n^{\eta}$, we have that: 
	
	$$
		m_3 = d((g_n^{\eta})^{-1},g) 
		= d(e, g_n^{\eta}g) 
		\geq d(e, g_n^{\eta}) 
		= d(e, (g_n^{\eta})^{-1}) 
		= m_4
	$$
	
	Let us now consider the sub-CGRs of $\lambda$ and $\gamma$ starting at $g$. Using  $\nu$-hyperbolicity, since $m_2 \geq m_1$, there exists $m_5 \geq m_3$ such that $d(\lambda(m_2), \gamma(m_5)) \leq 2 \nu$. Then by the triangle inequality and our above estimates, we have: 
	
	$$
		d(\gamma(m_4), \gamma(m_5)) \leq d(\gamma(m_4),e) + d(e, \lambda(m_2)) + d(\lambda(m_2), \gamma(m_5)) 
		\leq 6 \nu
	$$
	
	Therefore, 
	
$$
		d(e,g) = d(e, \gamma(m_3)) 
		\leq d(e, \gamma(m_4)) + d(\gamma(m_4), \gamma(m_3)) 
		\leq 2 \nu + d(\gamma(m_4),\gamma(m_5)) 
		\leq 8 \nu
$$
	
	where we have $d(\gamma(m_4), \gamma(m_3)) \leq d(\gamma(m_4),\gamma(m_5))$ because $m_5 \geq m_3 \geq m_4$. 
	
	\vs
	
	Thus,  $H_n(\eta) = g H_n(\theta)$ implies that $g $ is in the ball of radius $8 \nu$ about $e$ in $\text{Geo}((g_{\eta}^n)^{-1}, \sigma)$, which has cardinality at most $(2(8\nu + 2 \nu)+1)B = (20\nu + 1) B = K$. Thus, $F_n$-classes have cardinality at most $K$.

\end{proof}

The following remaining results have the same proof as in \cite{MS20}.

\begin{lem}
	\label{4.19}
	Let $n \in \mathbb{N}$. Then the map $H_n$ is Borel and so $\mathrm{im} H_n$ is analytic. 
\end{lem}

\begin{proof}
	The sets $\{(\eta, g_n^{\eta}) \in \partial \hat{\G} \times G\}, \{(\eta, T_n^{\eta}) \in \partial \hat{\G} \times 2^{G}\}$ and $G_{H_n} = \{(\eta, H_n(\eta)) : \partial \hat{\G} \times 2^G\}$ are all definable using formulas with countable quantifiers and references to the Borel sets $D$ and $S_n$ (see Claims \ref{4.12} and \ref{4.13}), so these sets are all Borel. As $G_{H_n}$ is the graph of $H_n$, it is Borel, so $H_n$ is Borel and hence $\text{im} H_n$ is analytic. 
\end{proof}

Using \cite[Lemma 2.3]{MS20}, there exists a finite Borel equivalence relation $F_n'$ on $2^G$ with $F_n \subseteq F_n'$. Since $F_n'$ is finite, Borel, there exists a Borel reduction $f_n : 2^G \to 2^{\mathbb{N}}$ from $F_n'$ to $ E_0$ for each $n \in \mathbb{N}$, using which we define $f: \partial \hat{\G} \to (2^{\mathbb{N}})^{\mathbb{N}}$ by $f(\eta) = (f_n(H_n(\eta)))_n$. Put $E' = f^{-1}(E_1)$, i.e. $\theta E' \eta \iff f(\theta) E_1 f(\eta)$.

\begin{lem}
\label{4.20}
The equivalence relation $E'$ is a hyperfinite countable Borel equivalence relation. 
\end{lem}

\begin{proof}
	
	Since $H_n$ is Borel, we have that $E'$ is Borel. We also have that $E'$ is hypersmooth by definition, and so it is hyperfinite by \cite[Theorem 8.1.5]{Gao}. We follow the same proof as the proof of \cite[Lemma 6.9]{MS20}, to show that $E'$ is countable.
	
	\vs
	
	For each $n \in \N$, define the relation $E_n'$ on $\partial \hat{\G}$ by $\eta E_n' \theta$ if $f_m(H_m(\eta)) = f_m(H_m(\theta))$ for all $m \geq n$. Each $E_n'$ is countable because if $\eta E_n' \theta$, then $f_n(H_n(\eta)) = f_n(H_n(\theta))$, and $f_n \circ H_n$ is countable-to-one since $H_n$ is countable-to-one because if $H_n(\eta) = H_n(\theta)$, then $\eta E \theta$ and $E$ is countable, and $f_n$ is finite-to-one since $F_n'$ is finite. Therefore, there are only countably many choices for $\eta$ such that $\eta E_n' \theta$ once $\theta$ is fixed. Thus, $E_n'$ is countable. Noting that $E' = \bigcup_{n \in \N} E_n'$, we obtain that $E'$ is countable.  
	
\end{proof}

\begin{lem}
	\label{4.21}
	$f$ is a homomorphism from $E \vert_Y$ to $E_1$.
\end{lem}

\begin{proof}
	
Suppose $\eta, \theta \in Y$ are $E$-related, as witnessed by $g \in G$ (so $g \eta = \theta$). By \cite[Theorem 5.9]{MS20} and \cite[Lemma 3.10]{MS20}, we have that $g \text{Geo}_1(e, \eta)$ and $\text{Geo}_1(e, \theta)$ differ by a finite set. By local finiteness of $\text{Geo}(e, \eta)$ and since $\eta, \theta \in Y$, we have that there exists $N \in \N$ such that $g T_n^{\eta} \subseteq \text{Geo}_1(e, \theta)$ for all $n \geq N$. By the proof of Lemma \ref{4.15}, we have $s_n^{\eta} = s_n^{\theta}$, which gives, together with $g T_n^{\eta} \subseteq \text{Geo}_1(e, \theta)$, that $gT_n^{\eta} = T_n^{\theta}$. This then yields $(g_n^{\theta})^{-1} g g_n^{\eta} H_n^{\eta} = H_n^{\theta}$ for all $n \geq N$. Thus, we have $H_n(\eta) F_n H_n(\theta)$ and so $H_n(\eta) F_n' H_n(\theta)$ for all $n \geq N$ since $F_n \subseteq F_n'$. Therefore, $f_n(H_n(\eta)) = f_n(H_n(\theta))$ for all $n \geq N$ and so $f(\eta) E_1 f(\theta)$. 
	
\end{proof}

Let us now establish Theorem \ref{A} on the hyperfiniteness of $E$, following the proof of \cite[Theorem A]{MS20}. 

\vs 

\textit{Proof of Theorem \ref{A}}.

\vs

Note that $E \vert_Y $ is a sub-relation of $ E'$. Indeed, if $\theta, \eta \in Y$ and $\theta E \eta$, then by Lemma \ref{4.21}, we have that $f(\theta) E_1 f(\eta)$, which implies $\theta E' \eta$. By Lemma \ref{4.20}, we have that $E'$ is hyperfinite, so $E \vert_Y$ is hyperfinite, since a sub-relation of a hyperfinite equivalence relation is hyperfinite. On $\partial \hat{\G} \setminus Y = [Z]_E$, $E$ is smooth by Lemma \ref{4.16}, and hence hyperfinite. Therefore, $E$ is hyperfinite on $\partial \hat{\G}$. 

\vs 

Recall that we worked with a fixed a finite generating set $X$ in this section. If we use a different finite generating set $X'$ for $G$, then the relative Cayley graph $\hat{\G}'$ corresponding to $X'$ is $G$-equivariantly quasi-isometric to $\hat{\G}$ (via the identity map on $G$), so $\partial \hat{\G}'$ is $G$-equivariantly homeomorphic to $\partial \hat{\G}$. It follows that the orbit equivalence relation of $G$ on $\partial \hat{\G}'$ is also hyperfinite. $\square$ 

\vs 

As a corollary, we obtain Corollary \ref{B}  on the hyperfiniteness of the action of $G$ on $\partial (G, \mathcal{P})$, where $\mathcal{P}$ is the collection of parabolic subgroups.

\vs 

\textit{Proof of Corollary \ref{B}}.

\vs

 By Theorem \ref{2.bb}, $\partial \hat{\G}$ embeds $G$-equivariantly and topologically into $\partial (G,\mathcal{P})$ with countable complement. Therefore, the orbit equivalence relation of $G$ on $\partial \hat{\G}$ is a subrelation of the orbit equivalence relation of $G$ on $\partial (G, \mathcal{P})$. Since the orbit equivalence relation of $G$ on $\partial \hat{\G}$ is hyperfinite (by Theorem \ref{A}) and since $\partial (G, \mathcal{P}) \setminus \partial \hat{\G}$ is countable, it follows that the orbit equivalence relation of $G$ on $\partial (G, \mathcal{P})$ is also hyperfinite. $\square$

	\bibliographystyle{plain} % We choose the "plain" reference style
	\bibliography{refs2}

\end{document}